\documentclass[11pt]{article}
\usepackage{amsmath,amssymb,amsthm}
\usepackage[margin=1.1in]{geometry}
\usepackage{booktabs}
\usepackage{enumitem}
\usepackage[colorlinks=true,linkcolor=blue,citecolor=blue,urlcolor=blue]{hyperref}

\newtheorem{theorem}{Theorem}[section]
\newtheorem{proposition}[theorem]{Proposition}
\newtheorem{lemma}[theorem]{Lemma}
\newtheorem{corollary}[theorem]{Corollary}
\newtheorem{conjecture}[theorem]{Conjecture}
\theoremstyle{remark}
\newtheorem{remark}[theorem]{Remark}
\numberwithin{equation}{section}

\newcommand{\R}{\mathbb{R}}
\newcommand{\N}{\mathbb{N}}

\title{A new arrangement inequality\\[2pt]\large (revised and extended version)}
\author{Mohammad Javaheri\\
Department of Mathematics and Statistics, Siena University \\
515 Loudon Road, Loudonville, NY 12221, USA\\
\texttt{mjavaheri@siena.edu}}

\begin{document}
\maketitle

\begin{abstract}
We study the cyclic inequality
\[
\sum_{i=1}^n x_i\,\sum_{i=1}^n x_i^a x_{i+1}^b\;\le\;\Big(\sum_{i=1}^n x_i^{(1+a+b)/2}\Big)^2,
\]
where $x_1,\dots,x_n\ge 0$ and indices are taken modulo $n$. Let $N(a,b)$ be the largest $n$ for which the inequality holds
for all nonnegative $x_i$. We show that $N(a,b)$ is bounded above by an explicit quantity
$N^*(a,b)$ obtained from the second variation at the constant vector, and we conjecture that
$N(a,b)=N^*(a,b)$. For $a=b=1$ we prove that the inequality holds if and only if $n\le 8$, which
settles the case left open in the original version of this paper. For $n=4$ the inequality holds
whenever $1,a,b$ are the sides of a triangle, with an explicit remainder term. We show that $N$ is
monotone along rays emanating from the line $a+b=1$; this corrects and extends an earlier
monotonicity statement. Finally, any $n$ nonnegative numbers can be arranged so that the
inequality holds; for $n\neq 6$ this is achieved by one of $n-1$ prescribed arrangements, which
yields a volume bound for every such $n$.
\end{abstract}


\section{Introduction and main results}

Throughout, $a,b\ge 0$, $x_1,\dots,x_n\ge 0$, indices are understood modulo $n$, and
\[
c=\frac{a+b+1}{2}.
\]
If $a+b=1$, the rearrangement inequality \cite{HLP} gives
$\sum_{i} x_i^a x_{i+1}^b\le \sum_i x_i$. As a generalization we consider
\begin{equation}\label{eq:main}
\sum_{i=1}^n x_i\;\sum_{i=1}^n x_i^a x_{i+1}^b\;\le\;\Big(\sum_{i=1}^n x_i^{c}\Big)^2 .
\end{equation}
Let $N(a,b)\in\N\cup\{\infty\}$ denote the largest $n$ such that \eqref{eq:main} holds for all
nonnegative $x_1,\dots,x_n$. By Proposition~\ref{prop:finite}, \eqref{eq:main} then holds for
every $n\le N(a,b)$ and fails for every $n>N(a,b)$, and $N(a,b)<\infty$ whenever $a+b\ne 1$.

The case $a=b=1$,
\begin{equation}\label{eq:11}
\sum_{i=1}^n x_i\;\sum_{i=1}^n x_ix_{i+1}\;\le\;\Big(\sum_{i=1}^n x_i^{3/2}\Big)^2,
\end{equation}
is the most interesting. The original version of this paper \cite{j} exhibited a counterexample with
$n=9$ and proved \eqref{eq:11} for $n\le 4$; a computer-based argument for $n=5,6,7$ was
reported in \cite{Hos}, leaving $n=8$ open. We settle the question completely.

\begin{theorem}\label{thm:N11}
Inequality \eqref{eq:11} holds for all nonnegative $x_1,\dots,x_n$ if and only if $n\le 8$.
That is, $N(1,1)=8$. For $n=8$, equality holds only when $x_1=\dots=x_8$.
\end{theorem}

The proof for $n=8$ (Section~\ref{sec:n8}) combines a spectral estimate for the cyclic sum
with a Lagrange-multiplier reduction. These reduce the problem to $44$ polynomial inequalities
in at most two variables, which we verify in exact rational arithmetic. The case $n\le 7$ also
follows from a shorter fully symmetric argument (Remark~\ref{rem:n7}).

The bound $n\le 8$ is not accidental. It is an instance of the following general upper bound,
obtained from the second variation of \eqref{eq:main} at the constant vector.

\begin{theorem}\label{thm:upper}
Let $a,b\ge0$ and $n\ge 2$. If
\begin{equation}\label{eq:threshold}
(a+b-1)^2\;>\;8ab\,\sin^2\!\frac{\pi}{n},
\end{equation}
then \eqref{eq:main} fails for some positive $x_1,\dots,x_n$. Consequently
$N(a,b)\le N^*(a,b)$, where
\[
N^*(a,b)=\max\Big\{n\ge 1:\ n=1 \text{ or } (a+b-1)^2\le 8ab\sin^2\frac{\pi}{n}\Big\}.
\]
Conversely, if the reverse strict inequality holds in \eqref{eq:threshold}, then the constant vector
is a strict local maximum, modulo scaling, of the ratio of the two sides of \eqref{eq:main}.
\end{theorem}

For $a=b=1$ we get $8\sin^2(\pi/n)\ge 1$ if and only if $n\le 8$, so $N^*(1,1)=8$, and
Theorem~\ref{thm:N11} says $N(1,1)=N^*(1,1)$. Extensive numerical experiments
(Section~\ref{sec:conj}) support the following.

\begin{conjecture}\label{conj:main}
$N(a,b)=N^*(a,b)$ for all $a,b\ge 0$ with $a+b\ne 1$.
\end{conjecture}

Note that $N^*(a,b)\ge 4$ if and only if $\sqrt a,\sqrt b,1$ satisfy the triangle inequality
(Lemma~\ref{lem:sqrt-triangle}). This region is strictly larger than the triangle region
\begin{equation}\label{eq:triangle}
b\le a+1,\qquad a\le b+1,\qquad 1\le a+b,
\end{equation}
in which \eqref{eq:main} holds for $n=4$ (Proposition~\ref{prop:n4}).

\begin{theorem}\label{thm:rays}
Let $a,b\ge 0$.
\begin{enumerate}[label=(\roman*)]
\item If $a+b\ge 1$, then $\lambda\mapsto N(\lambda a,\lambda b)$ is non-increasing on $[1,\infty)$.
\item If $a+b\le 1$, then $\lambda\mapsto N(\lambda a,\lambda b)$ is non-decreasing on $(0,1]$.
\end{enumerate}
In particular $N(u,u)$ is non-increasing on $[1/2,\infty)$ and non-decreasing on $(0,1/2]$.
\end{theorem}

\begin{remark}
The original version conjectured that $N(a,b)\to\infty$ as $a+b\to 1$. This is false without
further restrictions. If $a=0$ and $b\ne1$, then \eqref{eq:main} reads
$\sum x_i\sum x_i^b\le(\sum x_i^{(1+b)/2})^2$, which is the reverse of the Cauchy--Schwarz
inequality. It fails for $n=2$ whenever $x_1\ne x_2$, so $N(0,b)=1$ even as $b\to 1$.
Theorem~\ref{thm:upper} shows more generally that $N(a,b)=O\big(\sqrt{ab}/|a+b-1|\big)$, so
$N(a,b)\to\infty$ requires $|a+b-1|/\sqrt{ab}\to 0$. This does hold when $(a,b)\to$ the line
$a+b=1$ inside the region \eqref{eq:triangle}.
\end{remark}

Finally, \eqref{eq:main} can always be achieved by rearranging the variables.

\begin{theorem}\label{thm:arrangement}
Assume \eqref{eq:triangle}.
\begin{enumerate}[label=(\roman*)]
\item (Proposition~\ref{prop:arr}) Any $n$ nonnegative numbers can be arranged as
$x_1,\dots,x_n$ so that \eqref{eq:main} holds.
\item Let $R\subseteq[0,\infty)^n$ be a bounded, Lebesgue measurable set invariant under
permutation of coordinates, and let $\lambda(R)$ be the set of points of $R$ satisfying
\eqref{eq:main}. Then $\operatorname{vol}\lambda(R)\ge \operatorname{vol}R/(n-1)$ for every
$n\ne 6$, and $\operatorname{vol}\lambda(R)\ge \operatorname{vol}R/10$ for $n=6$.
\end{enumerate}
\end{theorem}

The original version proved (ii) only for prime $n$, and for general $n$ only the bound
$\operatorname{vol}R/n!$.

\section{Preliminaries}

\paragraph{Weighted AM--GM.} For $u,v,w\ge0$ and $r,s,t\ge0$ with $r+s+t=1$,
\begin{equation}\label{eq:amgm}
u^rv^sw^t\le ru+sv+tw.
\end{equation}
If $r,s,t>0$, equality holds iff $u=v=w$. If $t=0$ and $r,s>0$, equality holds iff $u=v$.

\paragraph{Power means.} For $t>0$ put $\sigma_t=\frac1n\sum_i x_i^t$. The map $t\mapsto\sigma_t^{1/t}$
is non-decreasing \cite[Thm.~16]{HLP}. The map $t\mapsto\log\sigma_t$ is convex on
$(0,\infty)$ by H\"older's inequality \cite{HLP}.

\paragraph{Karamata's inequality.} If $(\alpha_1,\dots,\alpha_m)$ majorizes
$(\beta_1,\dots,\beta_m)$ and $\varphi$ is convex, then
$\sum\varphi(\alpha_i)\ge\sum\varphi(\beta_i)$ \cite{MOA}.

\section{The case $n=4$}

\begin{proposition}\label{prop:n4}
Assume \eqref{eq:triangle}. Then for all $x,y,z,t\ge0$, with $S=x+y+z+t$,
\begin{equation}\label{eq:n4}
S\,(x^ay^b+y^az^b+z^at^b+t^ax^b)\;\le\;(x^c+y^c+z^c+t^c)^2-\Big(1-\frac1c\Big)\big(x^c-y^c+z^c-t^c\big)^2.
\end{equation}
In particular \eqref{eq:main} holds for $n=4$. Equality in \eqref{eq:main} holds iff
$\{a,b\}=\{0,1\}$ or $x=y=z=t$.
\end{proposition}

\begin{proof}
Conditions \eqref{eq:triangle} say exactly that $c\ge a$, $c\ge b$ and $c\ge 1$. Apply
\eqref{eq:amgm} with $u=(yz)^c$, $v=(xz)^c$, $w=(xy)^c$ and weights $r=1-a/c$, $s=1-b/c$,
$t=1-1/c$. The weights are nonnegative and sum to $3-(a+b+1)/c=1$. Comparing exponents gives
\begin{equation}\label{eq:p1}
x^ay^bz\le \Big(1-\frac ac\Big)(yz)^c+\Big(1-\frac bc\Big)(xz)^c+\Big(1-\frac1c\Big)(xy)^c,
\end{equation}
and likewise with $z$ replaced by $t$:
\begin{equation}\label{eq:p1b}
x^ay^bt\le \Big(1-\frac ac\Big)(yt)^c+\Big(1-\frac bc\Big)(xt)^c+\Big(1-\frac1c\Big)(xy)^c .
\end{equation}
Applying \eqref{eq:amgm} (with third weight $0$) to $u=x^{2c}$, $v=(xy)^c$ with weights
$1-\frac bc,\frac bc$, and to $u=y^{2c}$, $v=(xy)^c$ with weights $1-\frac ac,\frac ac$, gives
\begin{equation}\label{eq:p2}
x^{a+1}y^b\le \Big(1-\frac bc\Big)x^{2c}+\frac bc(xy)^c,\qquad
x^ay^{b+1}\le \Big(1-\frac ac\Big)y^{2c}+\frac ac(xy)^c .
\end{equation}
Adding \eqref{eq:p1}, \eqref{eq:p1b} and \eqref{eq:p2}, with
$S=x+y+z+t$,
\begin{multline*}
Sx^ay^b\le\Big(1-\frac bc\Big)x^{2c}+\Big(1-\frac ac\Big)y^{2c}+\Big(4-\frac3c\Big)(xy)^c\\
+\Big(1-\frac ac\Big)\big[(yz)^c+(yt)^c\big]+\Big(1-\frac bc\Big)\big[(xz)^c+(xt)^c\big].
\end{multline*}
Now sum this over the four cyclic shifts $(x,y,z,t)\to(y,z,t,x)$. Each of
$x^{2c},\dots,t^{2c}$ receives total weight $(1-\frac bc)+(1-\frac ac)=\frac1c$. Each adjacent
product $(xy)^c,(yz)^c,(zt)^c,(tx)^c$ receives
$(4-\frac3c)+(1-\frac ac)+(1-\frac bc)=4-\frac2c$. Each diagonal product $(xz)^c,(yt)^c$
receives $2(1-\frac ac)+2(1-\frac bc)=\frac2c$. Writing
$P=x^c+z^c$ and $Q=y^c+t^c$, and using $\sum x^{2c}+2(xz)^c+2(yt)^c=P^2+Q^2$, we get
\[
S\sum_{\rm cyc}x^ay^b\le \frac1c(P^2+Q^2)+\Big(4-\frac2c\Big)PQ=(P+Q)^2-\Big(1-\frac1c\Big)(P-Q)^2,
\]
which is \eqref{eq:n4}.

\emph{Equality.} Suppose equality holds in \eqref{eq:main}. If $a=0$ then equality forces
$\sum x\sum x^b=(\sum x^c)^2$, so by Cauchy--Schwarz the vectors $(x,y,z,t)$ and
$(x^b,\dots,t^b)$ are proportional; hence $b=c=1$ or all variables are equal. The case $b=0$ is
symmetric. If $a,b>0$, then $c=a=b$ is impossible, so we may assume $c\ne b$. Equality in the
first bound of \eqref{eq:p2} and in its cyclic shifts forces $x^{2c}=(xy)^c$, $y^{2c}=(yz)^c$,
$z^{2c}=(zt)^c$ and $t^{2c}=(tx)^c$, whence $x=y=z=t$.
\end{proof}

\begin{remark}
The sharpened form \eqref{eq:n4} was already implicit in the original proof, which discarded the
nonpositive term $(\frac1c-1)(P-Q)^2$. Numerically, \eqref{eq:n4} holds with no violation on
$2\cdot10^5$ random samples of $(a,b,x,y,z,t)$.
\end{remark}

\section{Finiteness, the second variation, and an upper bound for $N(a,b)$}

\begin{proposition}\label{prop:finite}
If $a+b\ne1$, then $N(a,b)<\infty$. If \eqref{eq:main} holds for all nonnegative $(k+1)$-tuples,
then it holds for all nonnegative $k$-tuples. Hence \eqref{eq:main} holds exactly for
$n\le N(a,b)$.
\end{proposition}

\begin{proof}
\emph{Finiteness.} If $a=b=0$, then \eqref{eq:main} fails for $n=2$, $x=(1,2)$. Assume
$a+b>0$ and that \eqref{eq:main} holds for all $n$. Let $f>0$ be continuous and non-constant on
$[0,1]$ with $f(0)=f(1)$. Put $x_i=f(\frac{i-1}{n})$ and
$y_i=(x_i^ax_{i+1}^b)^{1/(a+b)}$, which lies between $x_i$ and $x_{i+1}$. By the
intermediate value theorem, $y_i=f(t_i)$ for some $t_i\in[\frac{i-1}n,\frac in]$. Riemann sums
and \eqref{eq:main} give
$\int f\int f^{a+b}\le(\int f^c)^2$. The Cauchy--Schwarz inequality gives the reverse, with
equality only if $f$ and $f^{a+b}$ are proportional. That is impossible since $a+b\ne1$ and $f$ is
non-constant.

\emph{Descent.} Let $y_1,\dots,y_k\ge0$, and let $S,A,P$ be $\sum y_i$,
$\sum y_i^ay_{i+1}^b$ and $\sum y_i^c$; normalize $P=1$. For each $i$, insert a second copy of
$y_i$ directly after $y_i$. The new $(k+1)$-tuple has sum $S+y_i$, cyclic sum $A+y_i^{a+b}$
(the pair $(y_i,y_{i+1})$ survives and a pair $(y_i,y_i)$ is added), and $c$-power sum
$1+y_i^c$. Hence
\[
SA+Sy_i^{a+b}+Ay_i+y_i^{2c}\le 1+2y_i^c+y_i^{2c}.
\]
The terms $y_i^{2c}$ cancel. Summing over $i$ gives $kSA+S\sum y_i^{a+b}+AS\le k+2$. Since
$y_i^ay_{i+1}^b\le\frac{a}{a+b}y_i^{a+b}+\frac{b}{a+b}y_{i+1}^{a+b}$, we have
$A\le\sum y_i^{a+b}$, hence $(k+2)SA\le k+2$, i.e.\ $SA\le P^2$.
\end{proof}

\begin{proof}[Proof of Theorem~\ref{thm:upper}]
Let $x_i=1+\varepsilon v_i$ with $v\in\R^n$, and write $V_1=\sum v_i$, $V_2=\sum v_i^2$ and
$C=\sum v_iv_{i+1}$. Expanding $(1+\varepsilon v)^s=1+s\varepsilon v+\binom s2\varepsilon^2v^2+O(\varepsilon^3)$,
\begin{align*}
\textstyle\sum x_i\sum x_i^ax_{i+1}^b-\big(\sum x_i^c\big)^2
=\varepsilon^2\Big[(a+b-c^2)V_1^2+n\Big(\big(\tbinom a2+\tbinom b2-c(c-1)\big)V_2+ab\,C\Big)\Big]+O(\varepsilon^3),
\end{align*}
where $\binom s2=s(s-1)/2$ for real $s$;
the first-order terms cancel because $1+a+b=2c$. Write $v=m\mathbf 1+w$ with $\sum w_i=0$.
A direct computation shows that the coefficient of $m^2$ vanishes, as it must by homogeneity.
The quadratic form therefore equals
\begin{equation}\label{eq:quadform}
n\Big[\Big(\tbinom a2+\tbinom b2-c(c-1)\Big)|w|^2+ab\sum_i w_iw_{i+1}\Big].
\end{equation}
Moreover
\[
\tbinom a2+\tbinom b2-c(c-1)=\frac{(a-b)^2-2(a+b)+1}{4}=\frac{(a+b-1)^2}{4}-ab .
\]
For $w$ with $\sum w_i=0$ we have $\sum w_iw_{i+1}\le\cos(2\pi/n)|w|^2$, with equality for
$w_i=\cos(2\pi i/n)$. So the maximum of \eqref{eq:quadform} over $|w|=1$ is
\[
n\Big[\frac{(a+b-1)^2}{4}-ab\big(1-\cos\tfrac{2\pi}{n}\big)\Big]=n\Big[\frac{(a+b-1)^2}{4}-2ab\sin^2\frac{\pi}{n}\Big].
\]
Under \eqref{eq:threshold} this is positive, so $x_i=1+\varepsilon\cos(2\pi i/n)$ violates
\eqref{eq:main} for small $\varepsilon>0$. Under the reverse strict inequality the form is negative
definite on $\{\sum w_i=0\}$, which is the stated local maximality. Since
$\sin^2(\pi/n)$ decreases in $n\ge2$, the set of $n\ge2$ satisfying the non-strict reverse of
\eqref{eq:threshold} is an initial segment, and the bound $N\le N^*$ follows.
\end{proof}

\begin{corollary}\label{cor:examples}
\begin{enumerate}[label=(\roman*)]
\item $N(1,1)\le 8$: the inequality \eqref{eq:11} fails for every $n\ge 9$ (cf.\ the
counterexample $(8.5,9,10,11.5,12,11.5,10,9,8.5)$ of the original version, which is a discrete
cosine profile).
\item $N(0,b)=1$ for every $b\ne1$.
\item $N(a,b)\le\max\big\{1,\ 2\pi\sqrt{2ab}/|a+b-1|\big\}$ whenever $a+b\neq 1$.
\end{enumerate}
\end{corollary}

\begin{proof}
(i) holds because $8\sin^2(\pi/9)<1\le8\sin^2(\pi/8)$. For (ii), $N^*(0,b)=1$. For (iii), if
$N^*=n\ge2$ then $\sin(\pi/n)\ge|a+b-1|/(2\sqrt{2ab})$; now use $\sin x\le x$.
\end{proof}

\begin{lemma}\label{lem:sqrt-triangle}
$N^*(a,b)\ge4$ if and only if $\sqrt a,\sqrt b,1$ satisfy the triangle inequality. In
particular \eqref{eq:triangle} implies $N^*(a,b)\ge4$.
\end{lemma}

\begin{proof}
$N^*\ge4$ iff $(a+b-1)^2\le 4ab$. We have
$4ab-(a+b-1)^2=2(ab+a+b)-(a^2+b^2+1)$, and this equals
$(\sqrt a+\sqrt b+1)(-\sqrt a+\sqrt b+1)(\sqrt a-\sqrt b+1)(\sqrt a+\sqrt b-1)$. If
\eqref{eq:triangle} holds, then $\sqrt a\le\sqrt{b+1}\le\sqrt b+1$ and similarly for $\sqrt b$,
and $\sqrt a+\sqrt b\ge\sqrt{a+b}\ge1$.
\end{proof}

\section{Proof of Theorem~\ref{thm:N11}}\label{sec:n8}

By Corollary~\ref{cor:examples}(i), \eqref{eq:11} fails for $n\ge9$. By
Proposition~\ref{prop:finite}, it suffices to prove \eqref{eq:11} for $n=8$. Throughout this
section $n=8$, and we write
\[
S=\sum x_i,\quad Q=\sum x_i^2,\quad P=\sum x_i^{3/2},\quad
E=\sum_{i\ \rm even}x_i,\quad O=\sum_{i\ \rm odd}x_i .
\]

\subsection{A spectral estimate}

\begin{lemma}\label{lem:spectral}
For all $x\in\R^8$,
\[
\sum_{i=1}^8x_ix_{i+1}\;\le\;\kappa\Big(Q-\frac{E^2+O^2}{4}\Big)+\frac{EO}{2},\qquad \kappa=\cos\frac\pi4=\frac{1}{\sqrt2}.
\]
\end{lemma}

\begin{proof}
Let $B=\frac12(T+T^{-1})$, where $T$ is the cyclic shift, so that $\sum x_ix_{i+1}=x^{\mathsf T}Bx$.
The eigenvalues of $B$ are $\cos(2\pi k/8)$, $k=0,\dots,7$. The span $W$ of the indicator
vectors $\mathbf 1_{\rm even}$ and $\mathbf 1_{\rm odd}$ is the sum of the eigenspaces for
$k=0,4$. It is $B$-invariant, so $W^\perp$ is $B$-invariant too, and on $W^\perp$ all
eigenvalues are at most $\cos(\pi/4)$. Write $x=\bar x+x'$ with
$\bar x=\frac E4\mathbf 1_{\rm even}+\frac O4\mathbf 1_{\rm odd}\in W$ and $x'\in W^\perp$. Then
$x^{\mathsf T}Bx=\bar x^{\mathsf T}B\bar x+x'^{\mathsf T}Bx'\le\frac{EO}2+\kappa|x'|^2$ and
$|x'|^2=Q-\frac{E^2+O^2}4$.
\end{proof}

Since $Q\ge(E^2+O^2)/4$ by Cauchy--Schwarz within each parity class, the right-hand side of
Lemma~\ref{lem:spectral} increases with $\kappa$. Fix the rational number
\[
\kappa_0=\frac{99}{140}>\frac1{\sqrt2}\qquad(\text{indeed }2\cdot99^2=19602>140^2=19600),
\]
and define, for $x\in[0,\infty)^8$,
\begin{equation}\label{eq:Phi}
\Phi(x)=P^2-S\Big[\kappa_0\Big(Q-\frac{E^2+O^2}{4}\Big)+\frac{EO}{2}\Big].
\end{equation}
By Lemma~\ref{lem:spectral}, $S\sum x_ix_{i+1}\le P^2-\Phi(x)$. So \eqref{eq:11} for $n=8$
follows from
\begin{equation}\label{eq:Phipos}
\Phi(x)\ge0\quad\text{for all }x\in[0,\infty)^8,\ \text{with equality only if }x_1=\dots=x_8 .
\end{equation}
Note that $\Phi$ is invariant under permutations of the even coordinates, under permutations of
the odd coordinates, and under exchanging the two classes.

\begin{remark}
The symmetric bound $\sum x_ix_{i+1}\le\kappa Q+(1-\kappa)S^2/8$, which uses only the
eigenvalue $k=0$ exactly, is not enough for $n=8$: the resulting symmetric inequality fails near
$x=(1,t,\dots,t)$, $t\approx0.21$. Keeping the mode $k=4$ exact repairs this. The value of
$\kappa_0$ cannot be raised much: \eqref{eq:Phipos} fails for $\kappa_0=0.72$.
\end{remark}

\subsection{Reduction to few distinct values}

\begin{lemma}\label{lem:lagrange}
Suppose $\Phi(y)\ge0$ for every $y\in[0,\infty)^8$ with the following property $(\mathcal C)$:
\begin{quote}
within each parity class the positive coordinates take at most two distinct values, and if
both classes contain two distinct positive values $\alpha\ne\beta$ and $\gamma\ne\delta$, then
$\sqrt\alpha+\sqrt\beta=\sqrt\gamma+\sqrt\delta$.
\end{quote}
Then $\Phi\ge0$ on $[0,\infty)^8$. If moreover $\Phi(y)=0$ for such $y$ only when $y$ is constant,
then the same holds for all $x$.
\end{lemma}

\begin{proof}
Fix $x\ge0$, $x\ne0$, and consider the set $K$ of $y\ge0$ with the same $E$, $O$ and $P$ as
$x$. Then $K$ is compact and nonempty. On $K$, $S=E+O$ is fixed, so $\Phi$ is a decreasing
function of $Q$ (as $\kappa_0>0$). Let $y^*$ maximize $Q$ on $K$, let
$I=\{i:y^*_i>0\}$, and restrict to the relatively open face $\{y_i>0\ (i\in I),\ y_i=0\
(i\notin I)\}$. On it, $y^*$ is a local maximum of $Q$ subject to the smooth constraints
$\sum_{I_{\rm even}}y_i=E$, $\sum_{I_{\rm odd}}y_i=O$ and $\sum_Iy_i^{3/2}=P$ (omit a
constraint whose index set is empty).

If the constraint gradients are linearly dependent, then $(\sqrt{y_i^*})_{i\in I}$ is constant on
$I_{\rm even}$ and on $I_{\rm odd}$, so $y^*$ has property $(\mathcal C)$. Otherwise there are
multipliers with $2y_i^*=\lambda_{\rm par(i)}+\frac32\mu\sqrt{y_i^*}$ for $i\in I$. Thus within
each class, $s=\sqrt{y_i^*}$ solves the quadratic $2s^2-\frac32\mu s-\lambda=0$. Hence there are
at most two positive values per class, and when there are two, their square roots sum to
$\frac34\mu$, the same for both classes. So $y^*$ satisfies $(\mathcal C)$, and
$\Phi(x)\ge\Phi(y^*)\ge0$.

For the equality statement: if $\Phi(x)=0$, then $\Phi(y^*)=0$, so $y^*$ is constant, so $E=O$
and $Q(y^*)=(E^2+O^2)/4$, which is the smallest value of $Q$ compatible with $E,O$. Then
$Q(x)\le Q(y^*)$ forces equality in Cauchy--Schwarz within each class, i.e.\ $x$ is constant on
each class with equal class sums; so $x$ is constant.
\end{proof}

\subsection{The finite verification}

Write each positive value as the square of its square root. Then $x_i=s_i^2$ and
$x_i^{3/2}=s_i^3$, and $\Phi$ becomes a homogeneous polynomial of degree $6$ in the
$s_i$ with rational coefficients. By symmetry each parity class has one of the following
shapes: $\mathsf Z$ (all four entries $0$); $\mathsf S(b)$ ($b$ copies of one positive value,
$4-b$ zeros, $1\le b\le4$); or $\mathsf D(b,c)$ ($b$ copies of $(1+p)^2$ and $c$ copies of
$(1-p)^2$, $4-b-c$ zeros, $b\ge c\ge1$, $p\in[-1,1]$). In the shape $\mathsf D$ we normalized the
sum of the two square roots to $2$, which is allowed by homogeneity and by the relation in
$(\mathcal C)$. Up to exchanging the classes this gives $44$ cases.
\begin{itemize}
\item $\mathsf Z\mathsf S$: a single rational number, which is positive.
\item $\mathsf Z\mathsf D$: a univariate polynomial in $p$, shown positive on $[-1,1]$ by Sturm's theorem.
\item $\mathsf S\mathsf S$: put the even value $=1$ and the odd value $=t^2$. This gives a
univariate polynomial in $t$, shown positive on $(0,\infty)$ by Sturm's theorem. When $b=b'=4$
it has the double root $t=1$, which is divided out first.
\item $\mathsf S\mathsf D$: variables $(p,t)\in[-1,1]\times[0,\infty)$, where the
single-valued class equals $t^2$. The region $t\ge2$ is mapped to $\tau=1/t\in[0,\frac12]$
after multiplying by $\tau^6$.
\item $\mathsf D\mathsf D$: variables $(p,q)\in[-1,1]^2$, one parameter for each class.
\end{itemize}

The bivariate cases are verified by the following two exact tools.

\emph{Bernstein bounds.} If $f(s,r)=\sum a_{\alpha\beta}s^\alpha r^\beta$ has bidegree
$(d_1,d_2)$, then on $[0,1]^2$
\[
f=\sum_{i,j}b_{ij}\binom{d_1}{i}s^i(1-s)^{d_1-i}\binom{d_2}{j}r^j(1-r)^{d_2-j},\qquad
b_{ij}=\sum_{\alpha\le i,\ \beta\le j}\frac{\binom i\alpha\binom j\beta}{\binom{d_1}\alpha\binom{d_2}\beta}a_{\alpha\beta}.
\]
Hence $\min_{ij}b_{ij}\le f$ on $[0,1]^2$ \cite{Far}. After an affine change of variables this
applies to any box. Boxes are bisected until every Bernstein coefficient is positive.

\emph{Local lemma.} Equality $\Phi=0$ occurs in exactly three bivariate cases:
$\mathsf S(4)\mathsf D(b,4-b)$ at $(p,t)=(0,1)$, and $\mathsf D(b,4-b)\mathsf D(b',4-b')$ at
$(p,q)=(0,0)$. In these cases we expand $f=\sum_{k=2}^6F_k$ into homogeneous parts about the
equality point and check that $F_0=F_1=0$. We then find a rational $m>0$ with
$F_2-m(u^2+v^2)$ positive semidefinite, and a dyadic $\rho$ with
$m>\sum_{k\ge3}\|F_k\|_1\rho^{k-2}$, where $\|F_k\|_1$ is the sum of the absolute values of the
coefficients. Since $|F_k(u,v)|\le\|F_k\|_1\max(|u|,|v|)^{k-2}(u^2+v^2)$, we get $f>0$ on the
punctured square $0<\max(|u|,|v|)\le\rho$. The remaining region is covered by Bernstein boxes.
The values of $\rho$ obtained lie between $1/256$ and $1/8$.

All computations use exact rational arithmetic. No floating point is involved. The script
\texttt{verify\_n8.py} (Appendix~\ref{app:code}) carries out all $44$ cases in about three seconds
and reports success for each; the largest case needs $142$ boxes. Together with
Lemmas~\ref{lem:spectral} and~\ref{lem:lagrange}, this proves \eqref{eq:Phipos}, hence
\eqref{eq:11} for $n=8$ with equality only for constant vectors. Proposition~\ref{prop:finite}
gives \eqref{eq:11} for $n\le8$, and this completes the proof of Theorem~\ref{thm:N11}.
\qed

\begin{remark}[A symmetric proof for $n\le7$]\label{rem:n7}
For $n\le 7$ the cruder bound $\sum x_ix_{i+1}\le\kappa_nQ+(1-\kappa_n)S^2/n$, with
$\kappa_n=\cos(2\pi/n)$, already suffices. By Proposition~\ref{prop:finite} it is enough to treat
$n=7$. Since $\cos(2\pi/7)<\frac58$ and the bound increases with $\kappa$, it suffices to prove
the inequality
\[
S\Big(\tfrac58 Q+\tfrac{3}{56}S^2\Big)\le P^2 ,
\]
which is fully symmetric. The Lagrange argument of Lemma~\ref{lem:lagrange}, now with only $S$ and $P$
fixed, reduces it to vectors consisting of $k$ ones, $l$ copies of $t^2$ and $7-k-l$ zeros. That
leaves $28$ univariate polynomial inequalities on $[0,\infty)$, each verified by Sturm's theorem
(\texttt{verify\_n8.py --n7}). This gives an independent proof of \eqref{eq:11} for $n\le7$,
confirming \cite{Hos}. For $n=8$ the symmetric approach fails (see the remark after
\eqref{eq:Phipos}).
\end{remark}

\section{Monotonicity along rays}\label{sec:rays}

\begin{proof}[Proof of Theorem~\ref{thm:rays}]
(i) We may assume $a+b>1$, since otherwise $N=\infty$. Let $\lambda>1$, $n=N(a,b)+1$, and let
$x\ge0$ violate \eqref{eq:main} for $(a,b)$:
\[
\textstyle\sum x_i\,\sum x_i^ax_{i+1}^b>\big(\sum x_i^q\big)^2,\qquad q=c=\frac{a+b+1}{2}.
\]
Put $y_i=x_i^{1/\lambda}$. Then $y_i^{\lambda a}y_{i+1}^{\lambda b}=x_i^ax_{i+1}^b$,
$\sum y_i=\sum x_i^{r}$ and $\sum y_i^{c'}=\sum x_i^{p}$, where
\[
r=\frac1\lambda,\qquad c'=\frac{\lambda(a+b)+1}{2},\qquad p=\frac{c'}{\lambda}=\frac{a+b}{2}+\frac{1}{2\lambda}.
\]
Now $2p+1=2q+r$. We have $q\ge1\ge r$ and $q\ge p$ (the latter because $\lambda\ge1$). Hence
$(q,q,r)$ majorizes $(p,p,1)$: the largest entries satisfy $q\ge\max(p,1)$, and the two largest
satisfy $2q\ge p+\max(p,1)$. Since $t\mapsto\log\sigma_t$ is convex, Karamata's inequality gives
$2\log\sigma_p+\log\sigma_1\le2\log\sigma_q+\log\sigma_r$, that is,
\begin{equation}\label{eq:karamata}
\Big(\sum x_i^p\Big)^2\sum x_i\le\Big(\sum x_i^q\Big)^2\sum x_i^r .
\end{equation}
Therefore
\[
\sum y_i\sum y_i^{\lambda a}y_{i+1}^{\lambda b}=\sum x_i^r\sum x_i^ax_{i+1}^b>\frac{\sum x_i^r\,(\sum x_i^q)^2}{\sum x_i}\ge\Big(\sum x_i^p\Big)^2=\Big(\sum y_i^{c'}\Big)^2,
\]
so \eqref{eq:main} fails for $(\lambda a,\lambda b)$ at the same $n$. By
Proposition~\ref{prop:finite}, $N(\lambda a,\lambda b)\le N(a,b)$.

(ii) The same substitution with $0<\lambda<1$ and $a+b<1$ works. Now $r\ge1$, $q\le1$ and
$p\ge q$. Moreover $p\le\frac12+\frac r2\le r$, so $r\ge\max(p,1)$, and
$q\le\min(p,1)$. Since there are three entries with equal sums, these two conditions are
exactly majorization of $(p,p,1)$ by $(q,q,r)$, and \eqref{eq:karamata} follows as before.
\end{proof}

For completeness we keep the following statement of the original version; its proof (by
compactness and the rearrangement inequality) is unchanged. For fixed $n$,
\[
A_n(a,b)=\sup\Big\{\sum x_i\sum x_i^ax_{i+1}^b-\Big(\sum x_i^c\Big)^2:\ \max_i x_i=1\Big\}\longrightarrow0\qquad(a+b\to1).
\]

\section{Arrangements}

\begin{proposition}\label{prop:arr}
Assume \eqref{eq:triangle}, and let $x_1,\dots,x_n\ge0$. Put $S=\sum x_i$ and
$T=\sum_i\sum_{j\ne i}x_i^ax_j^b$. Then
\begin{equation}\label{eq:ST}
ST\le (n-1)\Big(\sum x_i^c\Big)^2 .
\end{equation}
Consequently, for any family $\pi_1,\dots,\pi_m$ of cyclic orders of $\{1,\dots,n\}$ in
which every ordered pair $(i,j)$, $i\ne j$, occurs as a pair of consecutive elements exactly
$\Lambda$ times in total (so $m=\Lambda(n-1)$), some $\pi_k$ satisfies
$S\sum_tx_{\pi_k(t)}^ax_{\pi_k(t+1)}^b\le(\sum x_i^c)^2$.
\end{proposition}

\begin{proof}
We have $ST=n^3\sigma_1\sigma_a\sigma_b-n^2\sigma_1\sigma_{a+b}$. By monotonicity of
$\sigma_t^{1/t}$ and $a,b,1\le c$, we get $\sigma_1\sigma_a\sigma_b\le\sigma_c^{(1+a+b)/c}=\sigma_c^2$.
By Cauchy--Schwarz, $\sigma_c^2\le\sigma_1\sigma_{a+b}$. Hence
$ST\le(n^3-n^2)\sigma_c^2=(n-1)(n\sigma_c)^2$, which is \eqref{eq:ST}. Writing $A_\pi=\sum_tx_{\pi(t)}^ax_{\pi(t+1)}^b$, the covering property gives
$\sum_kA_{\pi_k}=\Lambda T$. Hence $\min_kSA_{\pi_k}\le S\Lambda T/m=ST/(n-1)\le(\sum x_i^c)^2$.
\end{proof}

Taking all $(n-1)!$ cyclic orders gives part (i) of Theorem~\ref{thm:arrangement}. For part (ii) we need small families.

\begin{lemma}\label{lem:hamilton}
For every $n\ge2$ with $n\ne4,6$, there are $n-1$ cyclic orders of $\{1,\dots,n\}$ in which every
ordered pair of distinct elements is consecutive exactly once. For $n=6$ there are $10$ cyclic
orders in which every ordered pair is consecutive exactly twice.
\end{lemma}

\begin{proof}
For odd $n$, Walecki's construction \cite{Luc} decomposes $K_n$ into $(n-1)/2$ Hamiltonian
cycles; traversing each in both directions gives $n-1$ directed Hamiltonian cycles that use
every arc once. For even $n\ge8$ this is Tillson's theorem on Hamiltonian decompositions of the
complete symmetric digraph \cite{Til}. For $n=2$ take the single cyclic order $(1,2)$. For
$n=6$, the five Hamiltonian cycles
\[
(1,2,3,4,5,6),\ (1,2,3,4,6,5),\ (1,3,5,2,4,6),\ (1,3,6,2,5,4),\ (1,4,2,6,3,5)
\]
use every edge of $K_6$ exactly twice; traversing each in both directions gives $10$ cyclic
orders using every arc exactly twice. (For $n=4,6$ no decomposition into $n-1$ directed
Hamiltonian cycles exists \cite{Til}.)
\end{proof}

\begin{proof}[Proof of Theorem~\ref{thm:arrangement}(ii)]
For $n=4$, $\lambda(R)=R$ by Proposition~\ref{prop:n4}. Otherwise let $\pi_1,\dots,\pi_m$ be
the family of Lemma~\ref{lem:hamilton}, where $m=n-1$, or $m=10$ if $n=6$. Regard each $\pi_k$
as a permutation and write $(x\circ\pi)_t=x_{\pi(t)}$. By Proposition~\ref{prop:arr}, every
$x\in R$ has some $k$ with $x\circ\pi_k\in\lambda(R)$, using that $R$ is permutation invariant.
Hence $R\subseteq\bigcup_k\pi_k^{-1}\cdot\lambda(R)$. Since permutations preserve Lebesgue
measure, $\operatorname{vol}R\le m\operatorname{vol}\lambda(R)$.
\end{proof}

\section{Numerical evidence for Conjecture~\ref{conj:main}}\label{sec:conj}

For each $(a,b)$ below we maximized the scale-invariant ratio
$\sum x_i\sum x_i^ax_{i+1}^b/(\sum x_i^c)^2$ over $[0,\infty)^n$ from many random starting points
(Nelder--Mead followed by BFGS), for $3\le n\le N^*(a,b)+1$. In every case the maximum equals
$1$, attained only at constant vectors, for $n\le N^*(a,b)$, and exceeds $1$ for
$n=N^*(a,b)+1$.

\begin{center}
\begin{tabular}{cccc}
\toprule
$(a,b)$ & $N^*(a,b)$ & largest ratio at $n=N^*$ & ratio at $n=N^*+1$\\
\midrule
$(1,1)$ & 8 & $1$ (proved) & $1.00034$\\
$(1.5,1.5)$ & 6 & $1$ & $1.0053$\\
$(2,2)$ & 5 & $1$ & $1.0050$\\
$(3,3)$ & 4 & $1$ (proved) & $1.00002$\\
$(1,2)$ & 6 & $1$ & $1.0134$\\
$(0.5,1.5)$ & 7 & $1$ & $1.0012$\\
$(2,2.5)$ & 5 & $1$ & $1.0154$\\
$(2.5,3)$ & 5 & $1$ & $1.0369$\\
$(0.3,1.2)$ & 10 & $1$ & $1.00019$\\
$(1.2,0.6)$ & 9 & $1$ & $1.0011$\\
\bottomrule
\end{tabular}
\end{center}

For $(1,2)$, \eqref{eq:threshold} holds with equality at $n=6$; there the quadratic form
\eqref{eq:quadform} is only semidefinite, yet the inequality still holds. The conjecture is
proved for $(1,1)$ (Theorem~\ref{thm:N11}), for $a=0$ or $b=0$ (Corollary~\ref{cor:examples}), and
for every $(a,b)$ satisfying \eqref{eq:triangle} with $N^*(a,b)=4$ (by
Proposition~\ref{prop:n4}), e.g.\ $(3,3)$.

Taking $n=4$, Conjecture~\ref{conj:main} and Lemma~\ref{lem:sqrt-triangle} predict that
\eqref{eq:main} holds for $n=4$ exactly when $\sqrt a,\sqrt b,1$ form a triangle. Numerically this
is confirmed at points outside \eqref{eq:triangle} such as $(0.2,1.5)$, $(0.1,1.3)$, $(1,2.6)$,
$(3,6)$ and $(0.05,0.8)$; the last has $a+b<1$. Note that \eqref{eq:triangle} remains the exact
condition for Theorem~\ref{thm:arrangement}(i): taking $x=(t,1,\dots,1)$ and letting
$n\to\infty$ shows that it is necessary.

The method of Section~\ref{sec:n8} (spectral splitting of the cyclic sum, Lagrange reduction,
exact verification) is a plausible route to further cases of the conjecture.

\appendix
\section{The verification script}\label{app:code}

The file \texttt{verify\_n8.py}, distributed with this paper, is a self-contained Python~3 program
that uses only the standard library (\texttt{fractions}, \texttt{math}). It
\begin{enumerate}[label=(\arabic*)]
\item builds, for each of the $44$ shapes of Section~\ref{sec:n8}, the polynomial $\Phi$ with
$\kappa_0=99/140$ in exact rational arithmetic;
\item verifies the univariate cases by Sturm sequences, dividing out the double root $t=1$
where it occurs;
\item verifies the bivariate cases by Bernstein subdivision, after excising a neighbourhood of the
equality point certified by the local lemma;
\item with the option \texttt{--n7}, verifies the $28$ univariate inequalities of
Remark~\ref{rem:n7}.
\end{enumerate}
The output lists every case with the method used and the number of boxes, and ends with
\texttt{n=8: 44 cases, all verified: True}.

\section*{Acknowledgement}
I would like to thank Omar Hosseini for bringing to my attention the case $a=b=1$ of the
inequality \eqref{eq:main}. I would also like to thank Professor G.~Bennett for reviewing the
article and making useful comments. This revised version is prepared by an AI assistant (Claude).


\begin{thebibliography}{9}
\bibitem{Far} R.~T.~Farouki, The Bernstein polynomial basis: a centennial retrospective,
\emph{Comput.\ Aided Geom.\ Design} \textbf{29} (2012), 379--419.
\bibitem{HLP} G.~H.~Hardy, J.~E.~Littlewood, G.~P\'olya, \emph{Inequalities}, 2nd ed.,
Cambridge University Press, 1988.
\bibitem{Hos} O.~Hosseini, private communication, Fall 2006.

\bibitem{j} M. Javaheri, A new arrangement inequality, J. Inequal. Pure Appl. Math. 7 (2006), no. 5, Article 162.
\bibitem{Luc} \'E.~Lucas, \emph{R\'ecr\'eations math\'ematiques}, vol.~II, Gauthier-Villars,
Paris, 1892 (Walecki's construction).
\bibitem{MOA} A.~W.~Marshall, I.~Olkin, B.~C.~Arnold, \emph{Inequalities: Theory of Majorization
and Its Applications}, 2nd ed., Springer, 2011.
\bibitem{PM} M.~H.~Protter, C.~B.~Morrey, \emph{A First Course in Real Analysis}, Springer, 1997.
\bibitem{Ste} J.~M.~Steele, \emph{The Cauchy--Schwarz Master Class}, Cambridge University Press,
2004.
\bibitem{Til} T.~W.~Tillson, A Hamiltonian decomposition of $K^*_{2m}$, $2m\ge8$,
\emph{J.\ Combin.\ Theory Ser.\ B} \textbf{29} (1980), 68--74.
\end{thebibliography}
\end{document}